\documentclass[11pt]{amsart}%
\usepackage{hyperref}
\usepackage{amsmath}%
\usepackage{amsfonts}%
\usepackage{amssymb}%
\usepackage{amsthm}
\usepackage{mathrsfs}
\usepackage{comment}
\usepackage{todonotes}
\usepackage{bbm}
\usepackage{amscd}
\usepackage{youngtab}
\usepackage{tikz}
\usepackage{bm}
\usepackage{enumitem}
\usetikzlibrary{arrows}
\usetikzlibrary{matrix}

\newcommand{\ph}{\varphi}

\newcommand{\R}{\mathbb{R}}

\newcommand{\C}{\mathbb{C}}
\newcommand{\Z}{\mathbb{Z}}

\newcommand{\mc}{\mathcal}

\newcommand{\sur}{\twoheadrightarrow}

\DeclareMathOperator{\abs}{Abs}
\DeclareMathOperator{\pre}{\mc{P}}
\DeclareMathOperator{\co}{\mc{C}}
\newcommand{\sgn}{\mathrm{sgn}}

\DeclareMathOperator{\Aut}{Aut}

\DeclareMathOperator{\rank}{rank}
\DeclareMathOperator{\codim}{codim}

\newtheorem{theorem}{Theorem}[section]
\newtheorem{def-prop}[theorem]{Definition-Proposition}
\newtheorem{prop}[theorem]{Proposition}
\newtheorem{conj}[theorem]{Conjecture}
\newtheorem{question}[theorem]{Question}

\newtheorem{cor}[theorem]{Corollary}

\theoremstyle{definition}

\theoremstyle{remark}
\newtheorem*{remark}{Remark}

\begin{document}

\title[On the Sperner property for the absolute order]{On the Sperner property for the absolute order on complex reflection groups}
\author{Christian Gaetz}
\address{Department of Mathematics, Massachusetts Institute of Technology, Cambridge, MA.}
\email{\href{mailto:gaetz@mit.edu}{gaetz@mit.edu}} 
\author{Yibo Gao}
\email{\href{mailto:gaoyibo@mit.edu}{gaoyibo@mit.edu}} 
\date{\today}

\begin{abstract}
Two partial orders on a reflection group $W$, the \emph{codimension order} and the \emph{prefix order}, are together called the \emph{absolute order} $\abs(W)$ when they agree.  We show that in this case the absolute order on a complex reflection group has the strong Sperner property, except possibly for the Coxeter group of type $D_n$, for which this property is conjectural.  The Sperner property had previously been established for the \emph{noncrossing partition lattice $NC_W$} \cite{Muhle, Reiner}, a certain maximal interval in $\abs(W)$, but not for the entire poset, except in the case of the symmetric group \cite{Harper2019}.  We also show that neither the codimension order nor the prefix order has the Sperner property for general complex reflection groups.
\end{abstract}

\maketitle

\section{Introduction} \label{sec:intro}

A ranked poset $P$ with rank decomposition $P_0 \sqcup P_1 \sqcup \cdots \sqcup P_r$ is \emph{$k$-Sperner} if no union of $k$ antichains of $P$ is larger than the union of the largest $k$ ranks of $P$ (see \cite{Engel} for an introduction to Sperner theory).  It is \emph{strongly Sperner} if it is $k$-Sperner for $k=1,2,...,r+1$. 

The \emph{absolute order} on a Coxeter group $W$ has two equivalent descriptions: one in terms of the reflection lengths of its elements, and the other in terms of the codimensions of their fixed spaces.  The maximal intervals $[\mathrm{id}, c]$ in $\abs(W)$, where $c$ is a \emph{Coxeter element}, are known as the \emph{noncrossing partition lattices} $NC_W$.  They appeared in work of Brady and Watt \cite{Brady} on $K(\pi,1)$'s for Artin braid groups and have been studied combinatorially by Reiner \cite{Reiner} and Armstrong \cite{Armstrong}, among others.  The posets $NC_W$ are known to be strongly Sperner \cite{Muhle, Reiner}; this paper follows recent work of Harper and Kim \cite{Harper2019} in studying the problem of whether the whole absolute order $\abs(W)$ is strongly Sperner. 

We choose to work in the setting of general complex reflection groups, rather than just Coxeter groups.  In this generality, the two orders (the \emph{prefix} and \emph{codimension} orders) do not always agree.  We reserve the term \emph{absolute order} for the situation when they do agree (these cases have been classified by Foster-Greenwood \cite{Foster-Greenwood}, see Proposition \ref{prop:codim-equals-prefix}).  It is a common feature of work in this area that the two orders are better-behaved when they agree; and Theorem \ref{thm:main-theorem} and Conjecture \ref{conj:type-D} below make the case that this is true in regard to the Sperner property.

Our main result follows; see Section \ref{sec:background} for background and definitions, Section \ref{sec:proof} for the proof, and Section \ref{sec:counterexamples} for examples showing that the strong Sperner property does not hold in general for either the prefix order or codimension order when they disagree.  Conjecture \ref{conj:type-D} is discussed in Section \ref{sec:type-D}.

\begin{theorem} \label{thm:main-theorem}
Let $W=W_1 \times \cdots \times W_k$ be a finite complex reflection group, where each $W_i$ is in the family $G(m,1,n)$ or is an irreducible Coxeter group of type other than type $D$; then $\abs(W)$ is strongly Sperner.
\end{theorem}

\begin{remark}
Theorem \ref{thm:main-theorem} in the case where all $W_i$ are symmetric groups follows from a result of Harper and Kim \cite{Harper2019}.  The type $B$ case was proven independently by Harper, Kim, and Livesay \cite{Harper2019b} while this paper was in preparation.
\end{remark}

\begin{conj} \label{conj:type-D} \text{}
\begin{itemize} 
    \item[\normalfont{(a)}] Let $W$ be the Coxeter group of type $D_n$, then $\abs(W)$ has a normalized flow with $\nu \equiv 1$ (see Section \ref{sec:normalized-flow}).  And, therefore,
    \item[\normalfont{(b)}] $\abs(W)$ is strongly Sperner for any finite complex reflection group $W$ such that the prefix order $\pre(W)$ is equal to the codimension order $\co(W)$.
\end{itemize}
\end{conj}

\section{Background and definitions} \label{sec:background}
\subsection{Complex reflection groups} \label{sec:reflection-groups}

For $V$ an $n$-dimensional complex vector space, a finite group $W \subset GL(V)$ is a \emph{complex reflection group} of \emph{rank $n$} if it is generated by its set of \emph{reflections} $T=\{w \in W \: | \: \dim(V^w)=n-1\}$, where $V^w$ denotes the fixed subspace of $w$.  We say $W$ is \emph{irreducible} if $V$ is irreducible as a representation of $W$.  Any finite reflection group is a product of irreducible reflection groups.  The finite irreducible complex reflection groups were famously classified by Shephard and Todd \cite{Shephard}.  They consist of the following groups:
\begin{itemize}
    \item A 3-parameter infinite family of groups $G(m,p,n)$ where $p | m$ and $n,m \geq 1$, with the exception that $G(2,2,2)$ is reducible.  $G(m,p,n)$ has rank $n$ except when $m=1$, in which case $G(1,1,n)\cong S_n$ has rank $n-1$.
    \item 34 exceptional groups usually denoted $G_4,G_5,...,G_{37}$.
\end{itemize}

Among these, the groups which can be realized over a \emph{real} vector space $V$ are exactly the finite Coxeter groups.  These too have a familiar classification into Cartan-Killing types:
\begin{itemize}
    \item Type $A_{n-1}$: the symmetric groups $S_n = G(1,1,n)$,
    \item Type $B_n$: the hyperoctahedral groups $(\Z/2 \Z) \wr S_n =G(2,1,n)$,
    \item Type $D_n$: the groups $G(2,2,n)$, index-2 subgroups of the hyperoctahedral groups,
    \item Type $I_2(m)$: the dihedral groups $G(m,m,2)$, and
    \item The exceptional finite Coxeter groups of types $H_3, H_4, F_4, E_6, E_7, E_8$ (which coincide with $G_{23},G_{30},G_{28},G_{35},G_{36}$ and $G_{37}$ respectively).
\end{itemize}

\subsection{The absolute order on a reflection group} \label{sec:absolute-order}

Given $W$ a reflection group, the \emph{reflection length} $\ell_R(w)$ of an element $w \in W$ is defined to be the smallest $k$ such that $w=t_1 \cdots t_k$ with each $t_i \in T$, where $T$ denotes the set of all reflections in $W$. In this case, we say $t_1\cdots t_k$ is a \textit{reduced word} or \textit{reduced decomposition} for $w$. The \emph{prefix order} $\pre(W)$ is the partial order on $W$ such that $u \leq v$ if and only if
\begin{equation} \label{eq:prefix-order}
\ell_R(u)+\ell_R(u^{-1}v) = \ell_R (v).
\end{equation}
$\pre(W)$ is a ranked poset with rank function $\ell_R$.

\begin{remark}
The prefix order should not be confused with the \emph{weak order} on $W$ when $W$ is a Coxeter group.  The weak order is defined by an equation similar to (\ref{eq:prefix-order}), but with the reflection length $\ell_R$ replaced by the more classical length $\ell$, which records the shortest decomposition of an element of $W$ as a product of \emph{simple} reflections.  The strong Sperner property for the weak order in type $A_n$ was previously established by the authors \cite{weak-order-sperner}.  Another related order, the (strong) Bruhat order, is known to be strongly Sperner for all Coxeter groups by work of Stanley \cite{Stanley1980}.
\end{remark}

The \emph{codimension order} $\co(W)$ is defined so that $u \leq v$ if and only if 
\begin{equation} \label{eq:codimension-order}
\codim(V^u)+\codim(V^{u^{-1}v})=\codim(V^v).
\end{equation}
where $V^w$ denotes the subspace of $V$ fixed by the action of $w$ given by the inclusion $W \subset GL(V)$.

It was first proven by Carter \cite{Carter} that when $W$ is a Coxeter group, we have $\ell_R(w)=\codim(V^w)$ for all $w \in W$, so that in particular $\co(W)=\pre(W)$.  Foster-Greenwood has classified the complex reflection groups for which this coincidence continues to hold:

\begin{prop}[Foster-Greenwood \cite{Foster-Greenwood}] \label{prop:codim-equals-prefix}
For $W$ an irreducible complex reflection group, $\co(W)=\pre(W)$ if and only if $W$ is a Coxeter group or is in the family $G(m,1,n)$.
\end{prop}

We follow Huang, Lewis, and Reiner's convention \cite{Huang} by using the term \emph{absolute order} to refer to the codimension and prefix orders when they agree (in other parts of the literature, ``absolute order" is used to refer to what we call the prefix order).

The functions $\ell_R(w)$ and $\codim(V^{w})$ are both \emph{subadditive}; this means that for any $u,v \in W$
\begin{align} \label{eq:subadditive}
\ell_R(uv) &\leq \ell_R(u) + \ell_R(v), \\
\codim(V^{uv}) & \leq \codim(V^u) + \codim(V^v).
\end{align}

Now, for a reducible reflection group $W \times W' \subset GL(V \oplus V')$ it is clear that $\ell_R((w,w'))=\ell_R(w)+\ell_R(w')$ and $\codim((V \oplus V')^{(w,w')})=\codim(V^w)+\codim(V^{w'})$.  This implies that for either the prefix or codimension order, if $u \leq v$ and $u' \leq v'$ then $(u,u') \leq (v,v')$.  The reverse implication then follows by subadditivity.  Together these facts imply:

\begin{prop} \label{prop:product-of-groups-has-product-order}
Let $W \times W' \subset GL(V \oplus V')$ be a reducible reflection group.  Then $\pre(W \times W') \cong \pre(W) \times \pre(W')$ and $\co(W \times W') \cong \co(W) \times \co(W')$.
\end{prop}

\subsection{Rank generating functions} \label{subsec:rank-functions}
For $P$ a finite ranked poset, we let 
\[
F(P,q)=\sum_{i=0}^{\rank(P)} |P_i| \cdot q^i
\]
denote the rank generating function of $P$.  It is known that for any complex reflection group $W$ the codimension generating function
\begin{equation} \label{eq:codimension-generating-function}
C(W,q)=\sum_{i=0}^{\rank(W)} |\{w \in W | \codim(V^w)=i\}| \cdot q^i
\end{equation}
is equal to $(1+e_1q)\cdots (1+e_nq)$ where the $e_i$'s are positive integer invariants of $W$ called the \emph{exponents} (see Solomon \cite{Solomon} for a uniform proof).  For general complex reflection groups, the rank of $w$ in $\co(W)$ is not necessarily equal to $\codim(V^w)$ (for example, Foster-Greenwood \cite{Foster-Greenwood} gives examples of elements in rank one of $\co(W)$ with fixed-space codimension two), so that $C(W,q) \neq F(\co(W),q)$ in general.  However, $\pre(W)$ is always ranked by $\ell_R$, and thus when $\pre(W)=\co(W)$ we have 
\begin{equation} \label{eq:rank-generating-function}
F(\abs(W),q)=\prod_{i=1}^{n} \left(1+e_iq \right).
\end{equation}
This fact demonstrates the common theme that both $\pre(W)$ and $\co(W)$ are more tractable when they agree.

\subsection{The normalized flow property} \label{sec:normalized-flow}
The main tool that we will be using to establish the Sperner property of a poset is the theory of normalized flows, developed by Harper \cite{Harper1974}, and we will be mainly following his notation in this section. 

Let $G=(V=A \sqcup B,E)$ be a bipartite graph, equipped with a weight function $\nu:V\rightarrow\R_{\geq0}$. We consider $\nu$ as a measure. Namely, for any subset $ X \subset V$, let $\nu(X):=\sum_{x\in X}\nu(x)$. A \textit{normalized flow} on $G$, with respect to $\nu$, is a map $f:E\rightarrow \R_{\geq0}$ defined on the set of edges of $G$, such that for any $a\in A$ we have 
\[
\sum_{b\in D(a)}f(a,b)=\nu(a)/\nu(A),
\]
and for any $b\in B$ we have 
\[
\sum_{a\in D(b)}f(a,b)=\nu(b)/\nu(B),
\] 
where $D(x)$ denote the set of neighbors of $x$. 

Now let $P$ be a ranked poset with rank decomposition $P_0 \sqcup P_1 \sqcup \cdots \sqcup P_r$, with a weight function $\nu:P\rightarrow\R_{\geq0}$. We say that $f:E\rightarrow\R_{\geq0}$ is a \textit{normalized flow} on $P$ with respect to $\nu$, if the restriction of $f$ to the bipartite graph consisting of $P_i$ and $P_{i+1}$ and the covering relations between them is a normalized flow for each $i$.  Normalized flows will be useful to us thanks to the following theorem:

\begin{theorem}[Corollary to Theorem III of \cite{Harper1974}]\label{thm:nfsperner}
If a ranked poset $P$ has a normalized flow with respect to the weights $\nu \equiv 1$, then $P$ is strongly Sperner.
\end{theorem}
An important advantage of using normalized flows is that they behave well under product. Let $P$ and $Q$ be two ranked posets with weight functions $\nu_P$ and $\nu_Q$ respectively. Their \textit{(Cartesian) product} $P\times Q$ is $\{(p,q):p\in P,\ q\in Q\}$ where the partial order is defined as $(p,q)\leq (p',q')$ if $p\leq p'$ in $P$ and $q\leq q'$ in $Q$, with a weight function $\nu_{P\times Q}((p,q))=\nu_P(p)\cdot \nu_Q(q)$. We say that a ranked poset $P$ with rank decomposition $P_0 \sqcup P_1 \sqcup \cdots \sqcup P_r$ is \textit{log-concave} with respect to a weight function $\nu$, if $\nu(P_i)^2\geq \nu(P_{i-1})\nu(P_{i+1})$ for all $i$.
\begin{theorem}[Theorem I.C of \cite{Harper1974}]\label{thm:nfproduct}
Let $P$ and $Q$ be two ranked posets that are log-concave with respect to weight functions $\nu_P$ and $\nu_Q$. If both of them have normalized flows, then their product also has a normalized flow and is log-concave.
\end{theorem}

Another useful property of normalized flow is described as ``the Fundamental Lemma" by Harper \cite{Harper1974}; we reformulate it here: 

\begin{theorem}[Lemma I.B of \cite{Harper1974}]\label{thm:fundlemma}
Let $\ph:P\rightarrow Q$, with weight functions $\nu_P$ on $P$ and $\nu_Q$ on $Q$, be a surjective, measure-preserving and rank-preserving map of ranked posets. If $Q$ admits a normalized flow with weights $\nu_Q$ and for each covering relation $q\lessdot q'$ on $Q$, the induced bipartite graph on $\ph^{-1}(\{q\})$ and $\ph^{-1}(\{q'\})$ admits a normalized flow with weights $\nu_P$, then $P$ admits a normalized flow with weights $\nu_P$.
\end{theorem}

\section{Proof of Theorem \ref{thm:main-theorem}} \label{sec:proof}

\subsection{The generalized symmetric groups $G(m,1,n)$}\label{sec:m1n}
We first review some background on the complex reflection groups $G(m,p,n)$. Any $w\in G(m,p,n)$ can be expressed in the form $w=[a_1,\ldots,a_n|\sigma]$ where each $a_i \in\Z/m\Z$ and $p$ divides $\sum_{i=1}^na_i$. Note that we always require $p|m$. Naturally, we can view such $w=[a_1,\ldots,a_n|\sigma]$ as an element in $GL(\C^n)$ that sends the $k^{th}$ coordinate vector $v_k$ to $\exp(\frac{2\pi\sqrt{-1}a_k}{m})v_{\sigma(k)}.$ Correspondingly, we can also think of such $w$ as a permutation on $(\Z/m\Z)\times[n]$ such that $w(b,k)=(b+a_k,\sigma(k))$. There are two types of reflections in $G(m,p,n)$, which we call \textit{type (1)} and \textit{type (2)}:
\begin{enumerate}
\item $\sigma=(i,j)$ is a transposition for some $i<j$, $a_i=-a_j$ and $a_k=0$ for $k\neq i,j$;
\item $\sigma=\mathrm{id}$, $p|a_i\neq0$ for some $i\in[n]$ and $a_k=0$ for $k\neq i$.
\end{enumerate}

Shi \cite{shi} gives explicit formulae for the reflection length $\ell_R$ in $G(m,p,n)$ and we will be using his results for the special case $p=1$. For $w=[a_1,\ldots,a_n|\sigma]\in G(m,1,n)$, we can write $\sigma=\sigma^{(1)}\cdots\sigma^{(r)}$ in disjoint cycle notation. For $i=1,\ldots,r$, define the \textit{sign} of the cycle $\sigma^{(i)}$ to be $\sgn(\sigma^{(i)})=\sum_{j\in\sigma^{(i)}}a_j\in\Z/m\Z$. Let $t_0(w)=\#\{i\in[r]:\sgn(\sigma^{(i)})=0\}$ be the number of cycles of $\sigma$ with sign 0.
\begin{prop}[Theorem 2.1 of \cite{shi}]\label{prop:lenm1n}
For $w\in G(m,1,n)$, 
\[
\ell_R(w)=n-t_0(w).
\]
\end{prop}

We note the following subword property of the prefix order of any complex reflection group, which follows from \cite{Huang} (or from \cite{Armstrong} in the case of Coxeter groups).
\begin{prop}\label{prop:subword}
If $w=t_1\cdots t_{\ell}$ is a reduced word for $w$, then any subword $u=t_{j_1}\cdots t_{j_k}$ where $1\leq j_1<\cdots<j_k\leq\ell$ is reduced and $u\leq w$ in $\pre(W)$.
\end{prop}

We define the \textit{claw poset} of order $n$, denoted $C_n$, to be the ordinal sum of a single element with an antichain of size $n-1$. In other words, elements of $C_n$ are given by $x_0,x_1,\ldots,x_{n-1}$ and their orders relations are defined to be $x_0<x_i$ for all $i=1,\ldots,n-1$. For $n\geq2$, $C_n$ is a ranked poset of rank 2, which clearly has a normalized flow (with respect to $\nu\equiv1$). We are now ready to state the main result of the section.

\begin{theorem}\label{thm:m1n}
For $W=G(m,1,n)$, $\abs(W)$ has a coarsening (on the same underlying set $W$) isomorphic to $C_m\times C_{2m}\times\cdots\times C_{nm}$.
\end{theorem} 
\begin{proof}
We partition all reflections of $G(m,1,n)$ into sets $T_1\sqcup\cdots\sqcup T_n$, where $T_j$ consists of the following reflections $t=[a_1,\ldots,a_k|\sigma]$:
\begin{enumerate}
\item $\sigma=(i,j)$ for some $i<j$, $a_i+a_j=0$ and $a_k=0$ for $k\neq i,j$, and 
\item $\sigma=\mathrm{id}$, $a_j\neq0$, $a_k=0$ for $j\neq k$.
\end{enumerate}
In $T_j$, there are $(j-1)m$ reflections of type (1) and $m-1$ reflections of type (2). So its cardinality is $|T_j|=jm-1$. 

We claim that for each $w\in G(m,1,n)$, there is a unique way to write $w=t_{i_1}t_{i_2}\cdots t_{i_k}$ such that $1\leq i_1<\cdots<i_k\leq n$ and $t_{i_j}\in T_{i_j}$. Moreover, such decomposition of $w$ is reduced, meaning $\ell_R(w)=k$. Proceed by induction on $n$. The claim is clear when $n=1$. Now assume $n\geq2$. All reflections in $T_1,\ldots,T_{n-1}$, viewed as permutations on $(\Z/m\Z)\times[n]$, keep $(b,n)$ fixed for any (or a specific) $b\in\Z/m\Z$, while all reflections in $T_n$ do not. Therefore, if $w$ fixes $(b,n)$, we cannot choose any $t_n\in T_n$, and induction hypothesis takes care of the rest of the argument. If $w$ doesn't fix $(b,n)$, we have to choose $t_n\in T_n$ such that $wt_n^{-1}$ fixes $(b,n)$. Let $w^{-1}(b,n)=(a+b,n')\neq(b,n)$ for some $a\in\Z/m\Z$, and all $b\in\Z/m\Z$. Then $t_n$ must map $(b,n)$ to $(a+b,n')$. It's not hard to see such $t_n=[a_1,\ldots,a_n|\sigma]$ is unique: if $n'\neq n$, then $t_n$ must be of type (1) with $\sigma=(n',n)$, $a_n=a$, $a_{n'}=-a$; if $n'=n$, then $a\neq0$ and $t_n$ must be of type (2) with $a_n=a$. Then $w'=wt_n^{-1}$ fixes $(b,n)$ and lies in $G(m,1,n-1)$. By induction, the uniqueness and existence of such decomposition are established. To see that the expression obtained in this way is reduced, let's consider the signs of cycles in $w$, as in Proposition~\ref{prop:lenm1n}. If $t_n$ is of type (1), multiplication by $t_n^{-1}$ on $w$ splits the cycle containing $n$ into two cycles, keeping the overall sign. But as one of the cycles is a singleton containing $n$ with sign 0, we know that $t_0(w')=t_0(w)+1$ so $\ell_R(w')=\ell_R(w)-1$ by Proposition~\ref{prop:lenm1n}. If $t_n$ is of type (2), then in $w$, $n$ is a singleton with nonzero sign. After multiplication by $t_n^{-1}$, $n$ becomes a singleton with sign 0 so similarly, $\ell_R(w')=\ell_R(w)-1$. The rest of the argument follows by induction.

For each $j=1,\ldots,n$, label the unique minimum element of the claw poset $C_{jm}$ by $\mathrm{id}$ and its elements in rank 1 by the elements of $T_j$, recalling that $|T_j|=jm-1$. Every element in $P=C_{m}\times C_{2m}\times\cdots\times C_{nm}$ is then labeled by a tuple $x=(x_1,x_2,\ldots,x_n)$ where $x_j\in T_j\sqcup\{\mathrm{id}\}$. We identify this element as $w(x)=x_1x_2\cdots x_n\in G(m,1,n)$. By the uniqueness and existence argument in the last paragraph, we obtain a bijection between elements in $P$ and $W=G(m,1,n)$. Moreover, the rank of $x$ in $P$ equals the number of $x_j$'s in $T_j$, which is precisely $\ell_R(w(x))$ as shown above. Finally, $x'\leq x$ in $P$ precisely means that $x_1'\cdots x_n'$ is a subword of $x_1\cdots x_n$, after ignoring identity terms, which implies that $w(x')\leq w(x)$ by the subword property (Proposition~\ref{prop:subword}). As a result, $P$ is contained in $\abs(W)$ as desired.
\end{proof}
The following corollary is immediate.
\begin{cor} \label{cor:G(m,1,n)-normalized-flow}
The absolute order of $G(m,1,n)$ has a normalized flow with $\nu\equiv1$, and is thus strongly Sperner.
\end{cor}
\begin{proof}
Each $C_k$ is trivially log-concave and has a normalized flow, with $\nu\equiv1$. By Theorem~\ref{thm:nfproduct}, $C_m\times\cdots\times C_{nm}$ is log-concave and has a normalized flow. But $C_m\times\cdots\times C_{nm}\subset \abs(W)$, for $W=G(m,1,n)$ by Theorem~\ref{thm:m1n} so $\abs(W)$ has a normalized flow (taking the weights of all edges not contained in the product to be zero) and is thus strongly Sperner by Theorem~\ref{thm:nfsperner}.
\end{proof}

\subsection{The dihedral groups $G(m,m,2)$}

The following proposition was also stated without proof by Harper, Kim, and Livesay in their independent work \cite{Harper2019b}.  We provide a short proof here for the sake of completeness.

\begin{prop} \label{prop:dihedral-normalized-flow}
Let $W=G(m,m,2)$ for some $m \geq 2$, then $\abs(W)$ admits a normalized flow with $\nu \equiv 1$.
\end{prop}
\begin{proof}
The non-identity elements of $W$ are either reflections, with fixed-space codimension 1, or rotations, with fixed-space codimension 2.  The product of a reflection and a rotation is a reflection.  This implies that every reflection is covered by every rotation in the absolute order, so the Hasse diagram of $\abs(W)$ is a complete bipartite graph between both pairs of consecutive ranks, which clearly admits the desired normalized flow.
\end{proof}

\subsection{Exceptional Coxeter groups}

In order to verify the normalized flow property for the exceptional Coxeter groups, we take advantage of the large automorphism group of $\abs(W)$.  Since the set of reflections in $W$ is invariant under the conjugation action of $W$, it follows that the conjugation action of $W$ on $\abs(W)$ is by poset automorphisms.  

For $P$ a poset and $G \subset \Aut(P)$ a group of poset automorphisms, the \emph{quotient poset} $P/G$ has as elements the orbits of the action of $G$ on the set $P$.  For two orbits $\mc{O},\mc{O}'$ we have $\mc{O} \leq \mc{O}'$ in $P/G$ if and only if there exists some $x \in \mc{O}$ and some $x' \in \mc{O}'$ such that $x \leq_P x'$ (equivalently, for \emph{all} $x \in \mc{O}$ there exists such an $x' \in \mc{O}'$).

\begin{prop} \label{prop:flow-on-quotient}
Let $P$ be a finite ranked poset and let $G \subset \Aut(P)$, then $P$ has a normalized flow with $\nu_P \equiv 1$ if and only if $P/G$ has a normalized flow with $\nu_{P/G}(\mc{O})=|\mc{O}|$.
\end{prop}
\begin{proof}
First suppose that $P$ has a normalized flow with $\nu_P \equiv 1$.  Then it is immediate from the definitions that $P/G$ has a normalized flow with $\nu_{P/G}(\mc{O})=|\mc{O}|$ and edge weights 
\[
f_{P/G}(\mc{O},\mc{O}')=\sum_{\substack{x \in \mc{O} \\ y \in \mc{O}'}} f_P(x,y).
\]

Next, suppose that $P/G$ has a normalized flow with $\nu_{P/G}(\mc{O})=|\mc{O}|$.  Then the natural map $\ph: P \sur P/G$ given by $x \mapsto G \cdot x$ is clearly surjective, rank preserving, order preserving, and measure preserving.  Thus, by Theorem \ref{thm:fundlemma} it only remains to check that the induced subposet of $P$ on the elements $\ph^{-1}(\mc{O}) \cup \ph^{-1}(\mc{O}')$ admits a normalized flow with $\nu \equiv 1$ for any covering relation $\mc{O} \lessdot \mc{O}'$ in $P/G$.  We claim that the corresponding subgraph of the Hasse diagram of $P$ is in fact a biregular graph (regular on each side of the bipartition).  Indeed, let $x \in \mc{O}$ and suppose $y_1,...,y_k$ are the upper covers of $x$ which lie in $\mc{O}'$.  Given any other element $gx \in \mc{O}$ for $g \in G$, the upper covers of $gx$ which lie in $\mc{O'}$ are exactly $gy_1,...,gy_k$, and similarly for lower covers, so we have proven biregularity.  As observed by Harper \cite{Harper1974}, any biregular graph admits a normalized flow simply by taking all edge weights to be equal and scaled appropriately.  
\end{proof}

Proposition \ref{prop:flow-on-quotient} makes it feasible to construct normalized flows on $\abs(W)$ for the exceptional Coxeter groups of types $H_3, H_4, F_4, E_6, E_7,$ and especially $E_8$.  For example, the Coxeter group $W$ of type $E_8$ has approximately $7 \times 10^8$ elements, and the absolute order $\abs(W)$ has approximately $4 \times 10^{10}$ cover relations.  Explicitly checking for a normalized flow on a poset of this size is unfeasible; however $\abs(W)/W$ has only 112 elements and 449 cover relations, allowing for a simple computer check for a normalized flow using SageMath.  In this way explicit normalized flows (with $\nu(\mc{O})=|\mc{O}|$) have been constructed for $\abs(W)/W$ for all exceptional Coxeter groups $W$.  Together with Proposition \ref{prop:flow-on-quotient}, this proves that these posets $\abs(W)$ admit normalized flows (with $\nu \equiv 1$).  The above discussion proves the following proposition:

\begin{prop} \label{prop:exceptional-types-normalized-flow}
Let $W$ be an irreducible exceptional Coxeter group (type $H_3, H_4, F_4, E_6, E_7,$ or $E_8$), then $\abs(W)$ admits a normalized flow with $\nu \equiv 1$.  In particular, $\abs(W)$ is strongly Sperner.
\end{prop}

\subsection{Finishing the proof}

\begin{proof}[Proof Theorem \ref{thm:main-theorem}]

Let $W=W_1 \times \cdots \times W_k$ be as described in the statement of Theorem \ref{thm:main-theorem}.  By Corollary \ref{cor:G(m,1,n)-normalized-flow}, Proposition \ref{prop:dihedral-normalized-flow}, and Proposition \ref{prop:exceptional-types-normalized-flow}, the absolute order $\abs(W_i)$ of each factor admits a normalized flow with $\nu \equiv 1$.  By Proposition \ref{prop:product-of-groups-has-product-order}, 
\[
\abs(W) \cong \abs(W_1) \times \cdots \times \abs(W_k).
\]
By Equation (\ref{eq:rank-generating-function}) and the classical result that the coefficients of a real-rooted real polynomial are log-concave, we may apply Theorem \ref{thm:nfproduct} to see that $\abs(W)$ admits a normalized flow (with $\nu \equiv 1$), implying that it is strongly Sperner.
\end{proof}

\begin{remark}
It is important to note that proving the stronger normalized flow property for the irreducible groups $W_i$ is necessary for our proof of Theorem \ref{thm:main-theorem}.  It is not in general true that the product of (even rank log-concave) strongly Sperner posets is Sperner \cite{Proctor}.
\end{remark}

\section{Counterexamples} \label{sec:counterexamples}

In this section we show that neither $\co(W)$ nor $\pre(W)$ is Sperner for general complex reflection groups $W$.

\subsection{Codimension order}

\begin{prop} \label{prop:codimension-sperner-counterexample}
Many small examples show that $\co(W)$ need not be Sperner when $\co(W) \neq \pre(W)$.  For example, $\co(G(4,2,2))$ is not Sperner. 
\end{prop}
\begin{proof}
Let $P=\co(G(4,2,2))$, then $P$ has rank sizes $|P_0|=1,|P_1|=8,$ and $|P_2|=7$.  There are two maximal elements $x,y$ in rank one, thus $P_2 \cup \{x,y\}$ is an antichain of size 9, larger than any rank size.  
\end{proof}

In fact, $\co(W)$ need not even be \emph{ranked} in general: there is no consistent rank function for $\co(G(4,2,4))$.  This means that, for some complex reflection groups $W$ anyway, it does not even make sense to ask whether $\co(W)$ is Sperner.  We are not aware of an example $\co(W)$ which is ranked and Sperner, except in the cases $\co(W)=\pre(W)$.  

\begin{question}
Is there a complex reflection group $W$ such that $\co(W) \neq \pre(W)$, but $\co(W)$ is ranked and Sperner?
\end{question}

\subsection{Prefix order}

A finite poset is $\emph{graded}$ if all of its maximal chains have the same length.  Although the posets $\pre(W)$ are always ranked, they are not in general graded when $\pre(W) \neq \co(W)$.  Our strategy for finding a prefix order which is not Sperner is therefore to construct a reflection group $W$ such that $\pre(W)$ has a maximal element $m$ occurring in a rank below the largest rank.  The antichain consisting of $m$ together with the largest rank would then violate the Sperner property.

\begin{prop} \label{prop:prefix-counterexample}
The prefix order on the reflection group $G(10,5,3)^{12}$ is not Sperner.
\end{prop}
\begin{proof}
Let $W=G(10,5,3)$ and $P=\pre(W)$; the following facts are all easily verified by computer:
\begin{itemize}
    \item $\rank(P)=5$,
    \item $P$ has a maximal element $m$ of rank 3, and 
    \item $F(P,q)=1+33q+287q^2+519q^3+314q^4+48q^5$.
\end{itemize}

Although $P$ itself is strongly Sperner, we observe that the largest coefficient in $F(P,q)^{12}$ is the coefficient of $q^{37}$.  This means that the largest rank in $P^{12}$ is $(P^{12})_{37}$.  However $(m,...,m)$ is a maximal element in $P^{12}$ of rank 36, so $P^{12} \cong \pre(W^{12})$ is not Sperner. 
\end{proof}

\section{A conjecture for type $D_n$} \label{sec:type-D}
In this section, we discuss absolute order on the Weyl group $W$ of type $D_n$ for $n \geq 4$ ($W$ is also the group $G(2,2,n)$), and explain why the methods used to prove Theorem \ref{thm:m1n} will not work in this case.

It is well-known that the exponents $e_i$ for type $D_n$ are $1,3,5,\ldots,2n-3,n-1$. By the discussion in Section~\ref{sec:background}, we know that
\[
F(\abs(W),q)=(1+(n-1)q) \cdot \prod_{i=1}^{n-1}(1+(2i-1)q).
\]
Therefore it is natural to hope that $\abs(W)$ contains a a product of claw posets so that we might apply the techniques of Section~\ref{sec:m1n}. However, we will show that such strategy cannot work.

Following the setup from Section~\ref{sec:m1n}, every element in $G(2,2,n)$ can be written as $w=[a_1,\ldots,a_n|\sigma]$ with $a_k\in\Z/2\Z$ for $k\in[n]$, $\sum_{k=1}^n a_k=0$, and $\sigma\in S_n$. We can also view $w$ as a permutation on $1,2,\ldots,n,-1,-2,\ldots,-n$ (also called a \textit{signed permutation} on $[n]$) such that $w(i)=-w(-i)$ for all $i$, with the condition that there are an even number of $k\in[n]$ such that $w(k)<0$. For convenience, we write such $w$ in disjoint cycle notation for $\sigma$, and put a dot on $k\in[n]$ if $a_k=1\in\Z/2\Z$. For example, $(\dot1\dot2)$ refers to $w=[1,1,0,0,\ldots|(1,2)]\in G(2,2,n)$, which is also represented as $w(1)=-2$, $w(2)=-1$, $w(k)=k$ for $k=3,\ldots,n$, as a signed permutation.

We see that $G(2,2,n)$ doesn't have any reflections of type (2) (see Section~\ref{sec:m1n}). In fact, all of its reflections are of the form $(i,j)$ for some $i<j\in[n]$ or $(\dot i, \dot j)$ for some $i<j\in[n]$.
\begin{prop}\label{prop:DnoProd}
For $W=G(2,2,n)$ with $n\geq4$, $\abs(W)$ does not contain $C_2\times C_4\times\cdots\times C_{2n-2}\times C_{n}$.
\end{prop}
\begin{proof}
Assume for the sake of contradiction that $P=C_2\times C_4\times\cdots\times C_{2n-2}\times C_{n}\subset\abs(W)$. Each element of rank 1 in $P$ can be identified with a reflection in $W$, and the inclusion $P\subset\abs(W)$ partitions these elements into sets of size $1,3,5,\ldots,2n-3,n-1$. Denote these sets of reflections as $T_1,T_2,\ldots,T_n$ where $|T_k|=2k-1$ for $k=1,\ldots,n-1$ and $|T_n|=n-1$. Now, each element of $P$, and correspondingly each element of $W$, is labeled as $(t_1,t_2,\ldots,t_n)$ where $t_k\in T_k$ or $t_k=\mathrm{id}$. 

Focus on elements of rank 2. Each of them is labeled as $(t_1,t_2,\ldots,t_n)$, where all but two $t_k$'s are the identity. In particular, this element covers exactly two reflections in $P$, which belong to different $T_k$'s. Therefore, if $w\in\abs(W)$ has length 2, and covers exactly two reflections $t$ and $t'$ in $\abs(W)$, then $t\in T_k$, $t'\in T_{k'}$ for distinct $k\neq k'$. It is not hard to see that the following types of elements of $G(2,2,n)$ with absolute length 2 cover exactly two reflections:
\begin{itemize}
\item $tt'=t't$ with $t\in\{(i,j),(\dot i, \dot j)\}$ and $t'\in\{(i',j'),(\dot i',\dot j')\}$ such that $\{i,j\}\cap\{i',j'\}=\emptyset$;
\item $tt'=t't$ with $w=(\dot i)(\dot j)$, $t=(i,j)$ and $t'=(\dot i,\dot j)$.
\end{itemize}

Define a helper function $\phi:T\rightarrow{[n]\choose 2}$, where $T=T_1\sqcup\cdots\sqcup T_n$ is the set of reflections, by $\phi((i,j))=\phi((\dot i,\dot j))=\{i,j\}$. By arguments above, $(i,j)$ and $(\dot i,\dot j)$ cannot be in the same $T_k$, meaning that $\phi|_{T_k}$ is injective for each $k\in[n]$. Moreover, for distinct $t\neq t'\in T_k$, $\phi(t)$ and $\phi(t')$ must intersect. Thus, the image of $\phi(T_k)$, having the same cardinality as $T_k$, is a set of pairwise intersecting 2-element subset of $[n]$. When $n\geq4$, there can be at most $n-1$ such sets, meaning $|T_k|\leq n-1$. But $|T_{n-1}|=2n-3>n-1$, a contradiction.
\end{proof}

Proposition~\ref{prop:DnoProd} shows that the approach used in Theorem \ref{thm:m1n} doesn't work for the absolute order in type $D$. However, we still conjecture that the absolute order of type $D_n$ admits a normalized flow (Conjecture~\ref{conj:type-D}), and have verified this conjecture using computer search up to $n=8$.

\section*{Acknowledgements}
The authors wish to thank Richard Stanley for suggesting this problem, and Gene B. Kim and Larry Harper for helpful correspondence.  C.G. was partially supported by an NSF Graduate Research Fellowship under grant no. 1122374.

\bibliographystyle{plain}
\bibliography{main.bib}
\end{document}